\documentclass[12pt]{amsart}
\usepackage{graphicx}
\usepackage{amsfonts}
\usepackage{amssymb}
\usepackage{latexsym}
\usepackage{amscd}

\newtheorem{thm}{Theorem}[section]

\newtheorem{prop}[thm]{Proposition}

\newenvironment{pf*}[1]{\proof[#1]}{\endproof}
\usepackage{euscript}

\usepackage[OT2,OT1]{fontenc}

\newcommand{\cal}[1]{{\mathcal #1}}

\theoremstyle{definition}
\newtheorem{defn}{Definition}[section]

\theoremstyle{remark}

\newcommand{\diam}{\operatorname{diam}}
\newcommand{\dist}{\operatorname{dist}}

\renewcommand{\mod}{\operatorname{mod}}
\newcommand{\tl}{\tilde}
\newcommand{\wtl}{\widetilde}
\newcommand{\eps}{\epsilon}

\newcommand{\ceq}{\,\displaystyle{\Large\mbox{$\sim$}}_{\text{\hspace{-14pt}\tiny conf}}\,}

\newcommand{\Ccal}[1]{{{\cal{#1}}}}
\newcommand{\aaa}[1]{{{\mathbf{#1}}}}
\newcommand{\crit}{{{\aaa C}}}
\newcommand{\cu}{{\aaa C}_U}
\newcommand{\cur}{{({\aaa C}_U^n)^\RR}}

\newcommand{\hol}{{\aaa H}}
\newcommand{\mfld}{{\aaa M}}

\renewcommand{\k}{\kappa}

\renewcommand{\Im}{\operatorname{Im}}

\numberwithin{equation}{section}
\newcommand{\thmref}[1]{Theorem~\ref{#1}}

\newcommand{\cI}{{\cal I}}

\newcommand{\cG}{{\cal G}}

\newcommand{\cH}{{\cal H}}
\newcommand{\cR}{{\cal R}}

\newcommand{\cE}{{\cal E}}

\newcommand{\CC}{{\Bbb C}}
\newcommand{\RR}{{\Bbb R}}
\newcommand{\TT}{{\Bbb T}}
\newcommand{\ZZ}{{\Bbb Z}}
\newcommand{\NN}{{\Bbb N}}
\newcommand{\DD}{{\Bbb D}}
\newcommand{\HH}{{\Bbb H}}
\newcommand{\QQ}{{\Bbb Q}}

\newcommand{\cren}{\cR_{\text cyl}}

\renewcommand{\deg}{\operatorname{deg}}

\newcommand{\Aut}{\operatorname{Aut}}

\begin{document}
\addtolength{\evensidemargin}{-0.7in}
\addtolength{\oddsidemargin}{-0.7in}

\title
{Renormalization of unicritical analytic circle maps}
\thanks{This work was partially supported by NSERC Discovery Grant}
\author{Michael Yampolsky}
\date{\today}
\maketitle

\begin{abstract}
In this paper we generalize renormalization theory for analytic critical circle maps with a cubic critical point  to the case of maps with an arbitrary odd critical exponent by proving a quasiconformal rigidity statement for renormalizations of such maps.
\end{abstract}

\section{Introduction}
Critical circle maps are, along with unimodal maps of the interval, one of the two worked-out examples of renormalization and universality in dynamics. Renormalization theory has been developed for critical circle maps with critical exponent $n=3$. We refer the reader to \cite{Ya3} for a review of the history of the subject and the summary of the main conjectures, known as Landford's Program. These conjectures were settled by the author in \cite{Ya3,Ya4}. 

The purpose of this paper is to extend the renormalization theory to analytic critical circle maps with an arbitrary odd critical exponent
$n\in 2\NN+1$. All of the existing proofs apply to the case of arbitrary $n$ {\it mutatis mutandis} with a single notable exception: 
a result on non-existence of invariant Beltrami differentials for holomorphic extensions of renormalized maps (so-called holomorphic commuting pairs) proven by de~Faria in \cite{dF2}. de~Faria's argument used his construction of holomorphic commuting pairs for renormalizations of maps in the Arnold's family 
$$x\mapsto x+\theta+\frac{1}{2\pi}\sin(2\pi x)$$
(an archetypical family of critical circle maps with a cubic critical point). This proof thus does not generalize to an arbitrary odd $n$. 

In the present paper we present a different argument based on our complex {\it a priori} bounds \cite{Ya1} and (for each $n\in 2\NN+1$)
a construction of a quasiconformaly rigid family of Blaschke fractions which restrict to analytic homeomorphisms of the unit circle with a single critical point  of degree $n$. We thus extend the result on non-existence of invariant Beltrami differentials to the case of an arbitrary odd $n$ and complete the renormalization theory for this case. 

We note that after this paper has appeared as a preprint, we have learned that A.~Vieira in his recent PhD thesis \cite{vieira} has proposed a different approach to constructing quasiconformally rigid model holomorphic commuting pairs based on generalized Arnold's maps.

The structure of the paper is as follows. In the next section \S\ref{section preliminaries} we make the preliminary definitions, and describe the action of renormalization on unicritical circle maps. Our goals -- the renormalization hyperbolicity results for such maps -- are stated in \S~\ref{sec:results}. In \S~\ref{holomorphic pairs} we define holomorphic commuting pairs for an arbitrary odd critical exponent $n$ and state the 
main result on non-existence of invariant Beltrami differentials (Theorem~\ref{thm:main}). \S~\ref{sec:models} contains the construction of rigid Blaschke models. The proof of Theorem~\ref{thm:main} is given in \S~\ref{sec:proof}.

\section{Preliminaries}
\label{section preliminaries}

\noindent
{\bf Some notations.} We will assume that the reader is familiar with 
the renormalization theory of critical circle maps, and will only recall its main
points briefly. For a detailed discussion, we refer the reader to \cite{Ya3,Ya4}.

We use $\dist$ and $\diam$ to denote the Euclidean distance and diameter in $\Bbb C$.
We shall say that two real numbers $A$ and $B$ are $K$-commensurable for $K>1$ if
$K^{-1}|A|\leq |B|\leq K|A|$.
The notation $D_r(z)$ will stand for the Euclidean disk with the center at $z\in\Bbb C$ and
radius $r$. The unit disk $D_1(0)$ will be denoted $\Bbb D$. 
The plane $({\Bbb C}\setminus{\Bbb R})\cup J$
with the parts of the real axis not contained in the interval $J\subset \Bbb R$
removed  will be denoted ${\Bbb C}_J$.
By the circle $\TT$ we understand the affine manifold $\RR/\ZZ$, it is naturally identified
with the unit circle $S^1=\partial\DD$ via $\theta\mapsto e^{2\pi i\theta}$. The real translation $x\mapsto x+\theta$
projects to the rigid rotation by angle $\theta$, $R_\theta:\TT\to\TT$.
For two points $a$ and $b$ in the circle $\TT$
which are not diametrically opposite, $[a,b]$ will denote the shorter
of the two arcs connecting them. As usual, $|[a,b]|$ will denote the length of the 
arc. For two points $a,b\in \Bbb R$, $[a,b]$ will denote the closed interval with
endpoints $a$, $b$ without specifying their order.
The cylinder in this paper, unless otherwise specified will mean the affine manifold $\CC/\ZZ$.
Its equator is the circle $\{\Im z=0\}/\ZZ\subset \CC/\ZZ$.
A topological  annulus $A\subset \CC/\ZZ$ will be called an equatorial annulus, or an 
equatorial neighborhood, if it has a smooth boundary and contains the equator.
%%By ``smooth'' in this paper we will mean ``of class $C^\infty$'', unless another degree
%%of smoothness is specified. 

\begin{defn}
An analytic unicritical circle map is an analytic orientation preserving homeomorphism of
 $\TT$ with a single critical
point $c$. 
To fix our ideas, we will always place the critical
point of $f$ at $0\in\TT$. 
In what follows, we will refer to such maps as critical circle maps.
\end{defn}

The {\it critical exponent} of a critical circle map $f$ is the degree of the critical point $0$. Since $f$  is an analytic homeomorphism of the circle,
the critical exponent must be an odd integer.

As a homeomorphism of the circle, a critical circle map $f$ has a well-defined rotation number which we denote by $\rho(f)$.
It is useful to represent $\rho(f)$ as a finite or infinite contined fraction with positive terms
\begin{equation}
\label{rotation-number}
\rho(f)=\cfrac{1}{r_0+\cfrac{1}{r_1+\cfrac{1}{r_2+\dotsb}}}
\end{equation}
Further on we will abbreviate this expression as $[r_0,r_1,r_2,\ldots]$
for typographical convenience.  
Note that the numbers $r_i$ are determined uniquely if and only if $\rho(f)$ is
irrational. In this case we shall say that $\rho(f)$ (or $f$ itself) is of type bounded by $B$
if $\sup r_i\leq B$; it is of a periodic type if the sequence $\{r_i\}$ is periodic.
For a map $f$ such that the continued fraction (\ref{rotation-number})
contains at least $m+1$ terms, we denote $\{p_m/q_m\}$ the best rational 
approximation of $\rho(f)$ given  as the truncated continued fraction
$p_m/q_m=[r_0,r_1,\ldots,r_{m-1}]$, and set $I_m\equiv[0,f^{q_m}(0)]$. 
Recall that an iterate $f^k(0)$ is called a {\it closest return} of the critical point if
the arc $[0,f^k(0)]$ contains no iterates $f^i(0)$ with $i<k$.
One verifies then that the iterates $\{f^{q_m}(0)\}$ are closest returns of $0$.

%%By a classical result of Poincar{\'e} every circle homeomorphism $f$ with
%%an irrational rotation number is semi-conjugate to the rigid rotation $R_{\rho(f)}$.
%%Yoccoz \cite{Yoc} has shown that in the case when $f$ is a critical circle map,
%%the semiconjugacy actually becomes a conjugacy.

%%\subsection{Some function spaces}
%%In this section we will introduce some spaces of functions which will play a role in 
%%what follows. 
%%Firstly, we recall the definition of the Epstein class, introduced by Eckmann and Epstein 
%%\cite{EE}. An orientation preserving interval homeomorphism $g:I=[0,a]\to g(I)=J$ 
%%belongs to the {\it Epstein class $\Ccal E$} if it extends to an
%%analytic three-fold branched covering map of a topological disk
%% $G\supset I$ 
%%onto the double-slit plane ${\Bbb C}_{\tl J}$, where $\tl J\supset \cl J$.
%%The importance of $\cE$ will lie in its invariance under the renormalization of commuting pairs
%%(see below).
%%We refer the reader to \cite{Ya3} for a detailed discussion of $\cE$ and further
%%definitions. We will equip it with the Carath{\'e}odory topology. For $s\in(0,1)$ we will denote
%%$\cE_s$ a sequentially compact subset of $\cE$ consisting of maps with a geometric bound $s$.

Let us introduce some notation:
suppose, $W$ is a complex Banach space whose elements are functions of the complex variable. Let us say that
the {\it real slice} of $W$ is the real Banach space $W^\RR$ consisting of the real-symmetric elements of $W$.
If $X$ is a Banach manifold modelled on $W$ with the atlas $\{\Psi_\gamma\}$
we shall say that $X$ is {\it real-symmetric} if $\Psi_{\gamma_1}\circ\Psi_{\gamma_2}^{-1}(U)\subset W^\RR$ for any pair of indices $\gamma_1$, $\gamma_2$ and any open set $U\subset W^\RR$ on which this composition is defined. The {\it real slice of $X$} is then defined as the real
Banach manifold $X^\RR\subset X$ given by $\Psi_\gamma^{-1}(W^\RR)$ in a local chart $\Psi_\gamma$.
An operator $A$ defined on a subset of $X$ is {\it real-symmetric} if $A(X^\RR)\subset X^\RR$.

Denote $\pi$ the natural projection $\CC\to\CC/\ZZ$. For an equatorial annulus $U\subset \CC/\ZZ$
let ${\aaa A}_U$ be the space of bounded analytic maps $\phi:U\to \CC/\ZZ$ continuous up to the boundary, such that
$\phi(\TT)$ is homotopic to $\TT$, equipped with the uniform metric.
We shall turn ${\aaa A}_U$ into a real-symmetric complex Banach manifold as follows.
Denote $\tl U$ the lift $\pi^{-1}(U)\subset \CC$. The space of functions $\tl\phi:\tl U\to\CC$ which are analytic,
continuous up to the boundary, and $1$-periodic, $\tl\phi(z+1)=\tl\phi(z)$, becomes a Banach space when endowed 
with the sup norm. Denote that space $\tl{\aaa A}_U$. For a function $\phi:U\to\CC/\ZZ$ denote
$\check\phi$ an arbitrarily chosen lift so that $\pi(\check\phi(\pi^{-1}(z)))=\phi$.
Observe that $\phi\in{\aaa A}_U$ if and only if $\tl\phi=\check\phi-\operatorname{Id}\in\tl{\aaa A}_U$.
We use the local homeomorphism between $\tl{\aaa A}_U$ and ${\aaa A}_U$ given by
$$\tl\phi\mapsto \pi\circ(\tl\phi+\operatorname{Id})\circ\pi^{-1}$$
to define the atlas on  ${\aaa A}_U$. The coordinate change transformations
are given by $\tl\phi(z)\mapsto \tl\phi(z+n)+m$ for $n,m\in\ZZ$, therefore with
 this atlas  ${\aaa A}_U$ is a real-symmetric complex Banach manifold.

Fix ${n}\in 2\NN+1$.
Let $f:U\to\CC/\ZZ$ be an analytic critical circle map with critical exponent ${n}$. By definition, there is a neighborhood
of the equator in which $0$ is the only critical point of $f$.
Let $\tl f:\tl U\to\CC$ be a lift of $f$. The Argument Principle implies that for $\eps>0$ small enough,
if $\tl g\in\tl{\aaa A}_U$ is real-analytic, and $||\tl f-\tl g||<\eps$, and $g$ has the same order of the critical point as $f$ at $0$, 
then $g=\pi\circ(\tl g+\operatorname{Id})\circ\pi^{-1}$ is a critical circle map as well.

Let $\eps(f)$ be the supremum of such $\eps$,  and set
$$\cu^{n}=\cup_f \{\pi\circ(\tl g+\operatorname{Id})\circ\pi^{-1}|\;\tl g\in\tl {\aaa A}_U,
\tl g'(0)=-1,\;\tl g''(0)=0,\ldots,\tl g^{({n}-1)}(0)=0\text{ and }||\tl g-\tl f||<\eps(f)\}.$$
As shown in \cite{Ya3}, the space $\cu^{n}$ is a submanifold of ${\aaa A}_U$.
%%A tangent space to $\cu$ may be identified with the codimension $2$ Banach subspace of 
%%$\tl{\aaa A}_U$, given by 
%%  $${\aaa B}_U=\{\tl\phi\in\tl{\aaa A}_U\text{ with }\tl\phi'(0)=0,\tl\phi''(0)=0\}.$$

We shall say that $f$ is a {\it critical cylinder map} if $f\in\cu^{n}$ for some $U\subset \CC/\ZZ$.
Let us denote by $\cur$ the real slice of $\cu$:          
$$\cur=\{f\in\cu^{n}\text{ with }\bar f(z)=f(\bar z)\}=\{f\in{\aaa A}_U,\text{ such that }f|_\TT\text{ is a
critical circle map}\}.$$
                              
\section{Renormalization}
\label{sec:renorm}

\subsection*{Definition of renormalization of critical circle maps using commuting pairs.}
The strong analogy with universality phenomena in  statistical physics and with the 
already discovered Feigenbaum-Collett-Tresser universality in unimodal maps,
led the authors of \cite{FKS} and \cite{ORSS} to explain the existence of the universal
constants by introducing a renormalization operator acting on critical commuting pairs.
For simplicity of the exposition, we give the definitions only in the analytic case.

\begin{defn}
A  {\it critical  commuting pair} $\zeta=(\eta,\xi)$ consists of two 
analytic  orientation preserving interval homeomorphisms 
$\eta:I_\eta\to \eta(I_\eta),\;
\xi:I_{\xi}\to \xi(I_\xi)$, where
\begin{itemize}
\item[(I)] {$I_\eta=[0,\xi(0)],\; I_\xi=[\eta(0),0]$;}
\item[(II)] {Both $\eta$ and $\xi$ have analytic continuations to interval
neighborhoods of their respective domains  which commute, 
$\eta\circ\xi=\xi\circ\eta$;}
\item[(III)] {$\xi\circ\eta(0)\in I_\eta$;}
\item[(IV)] {$\eta'(x)\ne 0\ne \xi'(y) $, for all $x\in I_\eta\setminus\{0\}$,
 and all $y\in I_\xi\setminus\{0\}$;}
\item[(V)] $\eta'(0)=\xi'(0)=0.$
\end{itemize}
\end{defn}

\noindent
The critical exponent ${n}$ of a critical commuting pair is the order of the critical points of $\eta$ and $\xi$ at the origin; again this 
number must be an odd integer.

Let $f$ be a critical circle mapping, whose rotation number $\rho$
has a continued fraction expansion (\ref{rotation-number}) with
at least $m+1$ terms, and let $p_m/q_m=[r_0,\ldots,r_{m-1}]$. The pair of iterates $f^{q_{m+1}}$
and $f^{q_m}$ restricted to the circle arcs $I_m$ and $I_{m+1}$
correspondingly can be viewed as a critical commuting pair  in
the following way.
Let $\bar f$ be the lift of $f$ to the real line satisfying $\bar f '(0)=0$,
and $0<\bar f (0)<1$. For each $m>0$ let $\bar I_m\subset \Bbb R$ 
denote the closed 
interval adjacent to zero which projects down to the interval $I_m$.
Let $\tau :\Bbb R\to \Bbb R$ denote the translation $x\mapsto x+1$.
Let $\eta :\bar I_m\to \Bbb R$, $\xi:\bar I_{m+1}\to \Bbb R$ be given by
$\eta\equiv \tau^{-p_{m+1}}\circ\bar f^{q_{m+1}}$,
$\xi\equiv \tau^{-p_m}\circ\bar f^{q_m}$. Then the pair of maps
$(\eta|\bar I_m,\xi|\bar I_{m+1})$ forms a critical commuting pair
corresponding to $(f^{q_{m+1}}|I_m,f^{q_m}|I_{m+1})$.
Henceforth we shall  simply denote this commuting pair by
\begin{equation}
\label{real1}
(f^{q_{m+1}}|I_m,f^{q_m}|I_{m+1}).
\end{equation}
This allows us to readily identify the dynamics of
the above commuting pair with that of the underlying circle map,
at the cost of a minor abuse of notation.

Following \cite{dFdM1}, we say that the {\it height} $\chi(\zeta)$
of a critical commuting pair $\zeta=(\eta,\xi)$ is equal to $r$,
if 
$$0\in [\eta^r(\xi(0)),\eta^{r+1}(\xi(0))].$$
 If no such $r$ exists,
we set $\chi(\zeta)=\infty$, in this case the map $\eta|I_\eta$ has a 
fixed point.  For a pair $\zeta$ with $\chi(\zeta)=r<\infty$ one verifies directly that the
mappings $\eta|[0,\eta^r(\xi(0))]$ and $\eta^r\circ\xi|I_\xi$
again form a commuting pair.
For a commuting pair $\zeta=(\eta,\xi)$ we will denote by 
$\wtl\zeta$ the pair $(\wtl\eta|\wtl{I_\eta},\wtl\xi|\wtl{I_\xi})$
where tilde  means rescaling by the linear factor $\lambda={1\over \xi(0)}$.

\begin{defn}
The {\it renormalization} of a critical commuting pair $\zeta=(\eta,
\xi)$ is the commuting pair
\begin{center}
${\cal{R}}\zeta=(
\widetilde{\eta^r\circ\xi}|
 \widetilde{I_{\xi}},\; \widetilde\eta|\widetilde{[0,\eta^r(\xi(0))]}).$
\end{center}
\end{defn}

\noindent
The non-rescaled pair $(\eta^r\circ\xi|I_\xi,\eta|[0,\eta^r(\xi(0))])$ will be referred to as the 
{\it pre-renormalization} $p{\cal R}\zeta$ of the commuting pair $\zeta=(\eta,\xi)$.

For a pair $\zeta$ we define its {\it rotation number} $\rho(\zeta)\in[0,1]$ to be equal to the 
continued fraction $[r_0,r_1,\ldots]$ where $r_i=\chi({\cal R}^i\zeta)$. 
In this definition $1/\infty$ is understood as $0$, hence a rotation number is rational
if and only if only finitely many renormalizations of $\zeta$ are defined;
if $\chi(\zeta)=\infty$, $\rho(\zeta)=0$.

\noindent
For $\rho=[r_0,r_1,\ldots]\in [0,1]$ let us set 
$$G(\rho)=[r_1,r_2,\ldots]=\left\{\frac{1}{\rho}\right\},$$
where $\{x\}$ denotes the fractional part of a real number $x$
($G$ is usually referred to as the {\it Gauss map}).
As follows from the definition,
$$\rho({\Ccal R}\zeta)=G(\rho(\zeta))$$ for a real commuting pair $\zeta$ with
$\rho(\zeta)\ne 0$.

The renormalization of the real commuting pair (\ref{real1}), associated
to some critical circle map $f$, is the rescaled pair
$(\wtl{{f}^{q_{m+2}}}|\wtl{{I}_{m+1}},\wtl{{f}^{q_{m+1}}}|\wtl{{I}_{m+2}})$.
Thus for a given critical circle map $f$ the renormalization operator
 recovers the (rescaled) sequence of the first return maps:
$$\{ (\wtl{f^{q_{i+1}}}|\wtl{I_i},\wtl{f^{q_{i}}}|\wtl{I_{i+1}})\}_{i=1}^{\infty}.$$

The space of critical commuting pairs modulo affine conjugacy  will be denoted by $\crit$; its subset consisting
of pairs $\zeta$ with $\chi(\zeta)=r$ with $r\in\NN\cup\{\infty\}$ will be denoted by ${\aaa S}_r$.
For each ${n}\in 2\NN+1$ the subspaces of pairs with critical exponent ${n}$ will be denoted by $\crit^{n}$ and ${\aaa S}_r^{n}$
respectively.  

Let $C^0([0,1])$ be the Banach space of continuous functions of the interval $[0,1]$ with the uniform norm. We identify a pair $\zeta\in \crit^{n}$  with
a point in $$C^0([0,1])\times C^0([0,1])\times \RR_{>0}$$ via
$$\zeta=(\eta,\xi)\longrightarrow \left( \frac{1}{\xi(0)}\eta(\xi(0)x),\frac{1}{\eta(0)}\xi(\eta(0)x),\frac{\eta(0)}{\xi(0)}\right) .$$
This induces a metric on $\crit^{n}$, which we will call the $C^0$-metric.

Renormalization is a transformation $${\cal R}:\crit^{n}\setminus{\aaa S}^{n}_\infty\to\crit^{n}.$$
It is not difficult to show  that this transformation is injective on each ${\aaa S}_r^n$:

\begin{prop}[\cite{Yamp-towers}]
\label{injective1}Let $r\neq \infty$. Then
the map ${\cal R}:{\aaa S}_r^{n}\to\crit^{n}$ is one-to-one.
\end{prop}

\subsection*{Renormalization of maps of the cylinder.}

The cylinder renormalization operator $\cren$ was introduced in \cite{Ya3} as a more natural
way to renormalize analytic critical circle maps. It has two crucial advantages: the operator
acts on maps themselves, rather than their conjugacy classes; and extends to an analytic 
operator in the  Banach manifold $\cu^{n}$ of analytic critical circle 
maps defined in some equatorial neighborhood $U\subset \CC/\ZZ$. We briefly recall the definition
of $\cren$ below, the reader is directed to \cite{Ya3,Ya4} for the detailed exposition.

First note that in \cite{Ya3} we 
constructed a closed equivalence relation denoted $\ceq$ between analytic commuting pairs 
with the property that if $\zeta_1\ceq\zeta_2$ and $\chi(\zeta_1)=\chi(\zeta_2)$, then the analytic extensions of the commuting pairs
are conjugate in a specific neighborhood of the interval of definition by a conformal change of 
coordinates. In particular, $\rho(\zeta_1)=\rho(\zeta_2)$.

Furthermore, as we have shown, there exist an open set $\Omega\subset \crit^n$ and $N\in\NN$
such that for every $\zeta\in\crit^n$ with $\rho(\zeta)\notin\QQ$ the renormalizations $\cR^k\zeta\in\Omega$ for all $k$ larger than some $k_0=k_0(\zeta)$ such that the following holds.
For any $N$-times renormalizable $\zeta\in\Omega$ we have
$\cR^N\zeta\in\Omega$, and if $\zeta_1$, $\zeta_2$ are two $N$-times renormalizable pairs in $\Omega$ with $\zeta_1\ceq\zeta_2$, then
$$\cR^N\zeta_1\ceq\cR^N\zeta_2.$$ This last property implies, in particular,
that periodic orbits of $\cR^N$ project to the quotient space. 

The cylinder renormalization operator defined in \cite{Ya3} is 
a real-symmetric analytic operator $\cren$ from an open subset ${\aaa R}$ of $\cu^{n}$ to $\cu^{n}$
(thus $\cren$ maps an open subset of $$\mfld^{n}\equiv\cur$$ into $\mfld^{n}$),  such that the following holds.
%For every $f\in \mfld^{n}$ such that the continued fraction of $\rho(f)$ has at least $N+1$ terms,
%$f\in\aaa R$.
There exists $N\in\NN$ such that the operator $\cren$ changes the rotation number by shifting
its continued fraction expansion $N$ terms to the left.

For every $r\ne\infty$ there is a canonical homeomorphism onto the image
$${{\aaa S}^{n}_r}\cap\Omega/\ceq\overset{\iota_r}{\longrightarrow} \mfld^{n}$$ 
such that $\rho(\iota_r(\zeta))=\rho(\cR(\zeta))$, and $$\iota_r\circ\cR^N\equiv\cren\circ\iota_q$$
whenever both sides are defined.
 
Every map $f$ in ${\aaa R}$ posesses a domain $C_f$ called a fundamental crescent, which is 
bounded by two simple curves $l$ and $f^{q_N}(l)$ which meet at two endpoints $p_1$, $p_2$,
 which are
repelling fixed points of the iterate $f^{q_N}$. Moreover, the quotient of 
$\overline{C_f\cup f^{q_N}(C_f)}\setminus\{p_1,p_2\}$ by the iterate $f^{q_N}$ is conformally isomorphic to $\CC/\ZZ$. The first return map $R_f$ of $C_f$ under this isomorphism becomes $\cren f$.

\section{Statement of  the results}
\label{sec:results}

In this note we extend the general renormalization hyperbolicity results \cite{Yamp-towers,Ya4} to the case of unicritical circle maps
with an arbitrary critical exponent ${n}\in 2\NN+1$.
Before formulating our results let us make some further notation.
Let $\Sigma$ be the space of bi-infinite sequence of natural numbers, and 
denote by $\sigma:\Sigma\to\Sigma$ the shift on this space:
$$\sigma:(r_i)^\infty_{-\infty}\mapsto(r_{i+1})^\infty_{-\infty}.$$
Let us also complement the natural numbers with the symbol $\infty$ with the conventions
$\infty+x=\infty$, 
$1/\infty=0$; and denote by $\bar \Sigma$ the space $({\Bbb N}\cup\{\infty\})^{\Bbb Z}$,
and by $\bar \Sigma^+$ the space $({\Bbb N}\cup\{\infty\})^{\NN}$.
Our first theorem is a generalization of \cite{Yamp-towers} to the case of an arbitrary ${n}\in 2\NN+1$:
\begin{thm}[Renormalization horseshoe]
\label{existence of attractor}
Let ${n}\in 2\NN+1$.
There exists an $\cal R$-invariant  set ${\cal I}\subset\crit^{n}$ consisting of critical commuting pairs with critical exponent ${n}$ with irrational 
rotation numbers which has the following properties.
The action of $\cal R$ on $\cal I$ is topologically
conjugate to the two-sided shift $\sigma:\Sigma\to\Sigma$:
$$i\circ{\cal R}\circ i^{-1}=\sigma$$
and if $\zeta=i^{-1}(\ldots,r_{-k},\ldots,r_{-1},r_0,r_1,\ldots,r_k,\ldots)$ then
$\rho(\zeta)=[r_0,r_1,\ldots,r_k,\ldots]$. 
The set $\cal I$ is pre-compact, its closure ${\cal A}\subset \crit^{n}$ is the
horseshoe  attractor for the renormalization. That is, for any $\zeta\in \crit^{n}$ 
with irrational rotation number we have
$${\cal R}^m\zeta\to \cal A.$$
Moreover, for any two pairs $\zeta$, $\zeta'\in\crit^{n}$ with $\rho(\zeta)=\rho(\zeta')$
we have 
$$\dist({\cal R}^m\zeta,{\cal R}^m\zeta')\to 0.$$
\end{thm}

Furthermore, 
\begin{thm}[{\bf Hyperbolicity of renormalization}]
\label{renormalization hyperbolicity}
For each  $\zeta\in{\aaa S}_r^n\cap {\aaa R}$ we denote $\hat i(\zeta)\equiv i_r(\zeta)\in\mfld^{n}.$
The set 
$$\hat\cI\equiv \hat i(\cI)\subset\mfld^{n}$$
is uniformly hyperbolic for the analytic operator $\cren$, with one-dimensional
unstable direction.
\end{thm}

\noindent
The proofs of the above theorems for an arbitrary $n\in 2\NN+1$ are identical to the case $n=3$ except for a single missing piece, whose proof in \cite{dF2} uses the value $n=3$ in an essential way. 
We proceed to make the relevant definitions and formulate our main result, which completes the proofs of Theorems \ref{existence of attractor} and \ref{renormalization hyperbolicity} in the case of a general odd $n$, in the next section.

\section{Holomorphic commuting pairs and the statement of the main result}
\label{holomorphic pairs}
De~Faria \cite{dF2} introduced holomorphic commuting pairs to apply 
Sullivan's Riemann surface laminations technique to renormalization of 
critical circle maps. They are suitably defined holomorphic extensions 
of critical commuting pairs. Below we generalize the definition of \cite{dF2} to the case of an arbitrary odd critical exponent ${n}\in 2\NN+1$ in a straightforward fashion.
A critical commuting pair $\zeta=(\eta|_{I_\eta},\xi|_{I_\xi})\in\crit^{n}$  extends to a {\it
holomorphic commuting
pair} $\cal H$ if there exist four simply-connected $\RR$-symmetric quasidisks $\Delta$, $D$, $U$, $V$ such that
\begin{itemize}
\item  $\bar D,\; \bar U,\; \bar V\subset \Delta$,
 $\bar U\cap \bar V=\{ 0\}$;
 $U\setminus D$,  $V\setminus D$, $D\setminus U$, and $D\setminus V$ 
 are nonempty 
connected and simply-connected regions;
\item denoting $I_U\equiv U\cap\RR$, $I_V\equiv V\cap \RR$, $I_D\equiv D\cap\RR$ we have $I_U\supset I_\eta$, $I_V\supset I_{\xi}$;
\item the sets $U\cap\HH$, $V\cap \HH$, and $D\cap\HH$ are Jordan domains;
\item the maps $\eta$ and $\xi$ extend analytically to the domains $U$ and $V$ respectively, so that
 $\eta:U\to (\Delta\setminus \RR)\cup\eta(I_U)$ and
 $\xi:V\to(\Delta\setminus \RR)\cup\xi(I_V)$ are onto and
univalent;
\item $\nu\equiv \eta\circ\xi$ analytically contniues to an ${n}$-fold branch covering $\nu:D\to (\Delta\setminus \RR)\cup{\nu(I_D)}$ 
 with a unique critical point at zero.
\end{itemize}

\noindent
We shall call $\zeta$ the {\it commuting pair underlying $\cH$}, and write $\zeta\equiv \zeta_\cH$.
The domain $D\cup U\cup V$ of a  holomorphic commuting pair $\cH$
will be  denoted $\Omega$ or $\Omega_{\cal H}$, the range $\Delta$ will be sometimes denoted  $\Delta_\cH$.

We can associate to $\cH$ a piecewise-defined holomorphic map $S_\cH:\Omega\to\Delta$:
$$ S_{\cH}(z)=\left\{ 
\begin{array}{l}
\eta(z),\text{ if }z\in U;\\
\xi(z),\text{ if }z\in V;\\
\nu(z),\text{ if }z\in \Omega\setminus(U\cup V).\\
\end{array}
\right.,$$
which de~Faria \cite{dF2} called the {\it shadow} of $\cH$. Note that $S_\cH$ and $\cH$ share the same orbits as sets. The {\it filled Julia set } $K(\cH)$ is the set of points in $\Omega$ which do not escape $\Omega$ under the iteration of $S_{\cH}$; the {\it Julia set} $J(\cH)=\partial K(\cH)$. 

Denote $\hol^{n}$ the space of holomorphic commuting pairs with critical exponent ${n}$. It can be naturally endowed with Carath{\'e}odory topology (cf. \cite{Yamp-towers}).

It is easy to see directly from the definition (cf. \cite{dF2}) that:
\begin{prop}Let $\zeta$ be a commuting pair with $\chi(\zeta)<\infty$.
Suppose $\zeta$ is a restriction of a holomorphic commuting pair $\cH$, that is
$\zeta=\zeta_{\cH}$. Then there exists a holomorphic commuting pair $\cG$ with range $\Delta_\cH$,
such that $\zeta_\cG=\cR\zeta$.
\end{prop}

A standard pull-back argument (cf. \cite{dF2,Ya1}) gives:
\begin{thm}\label{qc conjugacy}
Suppose $\cH_1$ and $\cH_2$ are two holomorphic commuting pairs with the same critical exponent ${n}$ and the same irrational rotation number
$$\rho(\zeta_{\cH_1})=\rho(\zeta_{\cH_2})\notin \QQ.$$
Then $\cH_1$ and $\cH_2$ are quasiconformally conjugate.
\end{thm}

We let the modulus of a holomorphic commuting pair $\cH$, which we denote by $\mod(\cH)$ to be the modulus of the largest annulus $A\subset \Delta$,
which separates $\CC\setminus\Delta$ from $\overline\Omega$.

\begin{defn}\label{H_mu_def}
For $\mu\in(0,1)$ let $\hol^{n}(\mu)\subset\hol^{n}$ denote the space of holomorphic commuting pairs
${\cH}:\Omega_{{\cH}}\to \Delta_{\cH}$, with the following properties:
\begin{itemize}
\item $\mod (\cH)\ge\mu$;
\item {$|I_\eta|=1$, $|I_\xi|\ge\mu$} and $|\eta^{-1}(0)|\ge\mu$; %$\min(|I_\eta|,|I_\xi|)\ge\mu$; 
\item $\dist(\eta(0),\partial V_\cH)/\diam V_\cH\ge\mu$ and $\dist(\xi(0),\partial U_\cH)/\diam U_\cH\ge\mu$;
\item {the domains $\Delta_\cH$, $U_\cH\cap\HH$, $V_\cH\cap\HH$ and $D_\cH\cap\HH$ are $(1/\mu)$-quasidisks.}
\item $\diam(\Delta_{\cH})\le 1/\mu$;
\end{itemize}
\end{defn}

\noindent
We say that a real commuting pair $\zeta=(\eta,\xi)$ with
an irrational rotation number has
{\it complex {\rm a priori} bounds}, if there exists $\mu>0$ such that all renormalizations of $\zeta=(\eta,\xi)$ extend to 
holomorphic commuting pairs in $\hol^{n}(\mu)$.
The existense of complex {\it a priori} bounds is a key analytic issue 
 of renormalization theory:

\begin{thm}
\label{complex bounds}There exist universal constants $\mu=\mu({n})>0$ and $K=K({n})>1$ such that
the following holds. 
Let $\zeta\in\crit^{n}$ be a critical commuting pair with an irrational rotation number. 
Then there exists $N=N(\zeta)$ (which can be chosen uniformly in families of analytic pairs compact in the sense of Carath{\'e}odory topology, cf. \cite{Yamp-towers}) such that for all
$m\geq N$ the  commuting pair $\cR^m\zeta$ extends to a holomorphic commuting pair
$\cH_m:\Omega_m\to\Delta_m$ in $\hol^{n}(\mu)$.
The range $\Delta_m$ is a Euclidean disk.

Furthermore, let $\zeta_1$ and $\zeta_2$ be two critical commuting pairs with the same irrational rotation number and the same critical exponent ${n}$. There exists $M$ (which again can be chosen uniformly in compact families) such that for all $m\geq M$ the renormalizations 
$\cR^m\zeta_1$, $\cR^m\zeta_2$ extend to holomorphic commuting pairs $\cH_1^m$, $\cH_2^m$ which are $K$-quasiconformally conjugate.
\end{thm}

\noindent
We first proved this theorem in \cite{Ya1} for commuting pairs 
$\zeta$ in an Epstein class $\cE_s$ and ${n}=3$.
 Our proof was later adapted by de~Faria and de~Melo \cite{dFdM2} 
to the case of a non-Epstein critical commuting pair and ${n}=3$. The proof extends {\it mutatis mutandis} to the case of an arbitrary ${n}$.

We are now ready to formulate our main result:

\begin{thm}[{\bf Main Result}]
\label{thm:main}
Let $\cH\in\hol^{n}$ and $\rho(\zeta_\cH)\notin\QQ$. Then every $\RR$-symmetric, $S_{\cH}$-invariant Beltrami differential $\mu$ entirely supported on  $K(\cH)$ is trivial.
\end{thm}

\noindent
Theorems \ref{existence of attractor} and \ref{renormalization hyperbolicity} follow from the above statement immediately, as the remainder of their proof for $n=3$ works for an arbitrary odd value of $n$.
The proof of \thmref{thm:main} will occupy the rest of the paper.

\section{Blaschke models}
\label{sec:models}

Recall that the subgroup of $\Aut(\hat\CC)$ which preserves the unit circle $S^1$ consists of M\"obius transformations of the form 
\begin{equation}
\label{eq-mob}
f(z)=e^{2\pi i\theta} \frac{z+a}{1+\bar az},\text{ where }a\in\CC\setminus S^1\text{ and }\theta\in\RR.
\end{equation}
A {\it Blaschke fraction} is a rational function obtained as a product of finitely many terms of the form (\ref{eq-mob}); it is called a {\it Blaschke product} if $a\in\DD$ in each of these terms.

We note the following standard fact:
\begin{thm}
Suppose a rational map $B(z)$ preserves the unit circle $S^1$. Then it is a Blaschke fraction. If it preserves the unit disk $\DD$ then it is a Blaschke product.
\end{thm}

\noindent
 We claim:
\begin{thm}
\label{th-blaschke}
Let $m\in \NN$ and set $n=2m+1$. There exists a unique Blaschke fraction $B_n(z)$ of degree $n$ with the following properties:
\begin{enumerate}
\item $B_n(1)=1$ and $1$  is a critical point of $B_n$ with multiplicity $n$;
\item $B_n(0)=0$, $B_n(\infty)=\infty$, and $B_n(z)$ has critical points at $0$ and $\infty$, both with multiplicity $m+1$;
\item $B_n(z)$ has no critical points other than $0$, $1$, and $\infty$;
\item The restriction $B_n:S^1\to S^1$ is a homeomorphism.
\end{enumerate}
\end{thm}

\begin{proof}
Let us set
$$P(z)=\displaystyle\sum_{k=0}^m\left(\begin{array}{c} n\\k\end{array}\right)z^{n-k}.$$
The Binomial Formula gives:
$$(z-1)^n=P(z)-Q(z),\text{ where }Q(z)=z^n P\left(\frac{1}{z}\right).$$
We set
$$B_n(z)=\frac{P(z)}{Q(z)}.$$
An elementary estimate shows that $Q(1)\neq 0$; hence, $B_n(z)-1$ has a zero of degree $n$. By construction, $B_n$ has a critical point of order $m+1$ at infinity (and, by symmetry) at $0$, and both of them are fixed.  It has no other critical points by Riemann-Hurwitz theorem. 

We claim that $B_n$ is a Blaschke fraction. Indeed, for each non-zero real root $c$ of $P(z)$ the denominator has a factor $z(\frac{1}{z}-c)$ so that
$$\frac{z-c}{z(\frac{1}{z}-c)}=\frac{z-c}{1-cz}$$ is a term of the form (\ref{eq-mob}). For each pair of complex roots $\alpha$, $\bar\alpha$, combining with the corresponding factors in the denominator, we obtain a product of two M\"obius maps of the form (\ref{eq-mob}):
$$\frac{(z-\alpha)(z-\bar\alpha)}{z^2(\frac{1}{z}-\alpha)(\frac{1}{z}-\bar\alpha)}.$$

Claim (4) follows by considerations of topological degree.

Finally, let us prove the uniqueness of $B_n$. Suppose $F(z)$ is a different Blaschke fraction with the same properties. Note that
$F(z)-1$ has a zero of order $n$ at $1$. Since $\deg F=n$,
$$F(z)-1=\frac{(z-1)^n}{G(z)},$$
where $G(z)$ is a polynomial and $G(1)\neq 0$. Moreover, considerations of branching order at $\infty$ imply that $\deg G=m$. 
Denote $\bar G(z)$ the polynomial $\overline{G(\bar z)}$.
Using the symmetry $$F(z)=\overline{\left( F\left( \frac{1}{\bar z}\right)\right)^{-1}},$$
we obtain
$$\frac{(z-1)^n+G(z)}{G(z)}=\displaystyle\frac{z^n\bar G\left(\displaystyle\frac{1}{z}\right)}{(1-z)^n+z^n\bar G\left(\displaystyle\frac{1}{z}\right)}.$$
Hence, 
$$z^n\bar G\left(\frac{1}{z}\right)-G(z)=(z-1)^n.$$
Note that the first polynomial on the left-hand side has powers of $z$ from $m+1$ to $n$, and the second one (that is, $G(z)$) has powers of $z$ from $0$ to $m$. Thus, $G\equiv Q$ and the proof is complete.
\end{proof}

\section{Proof of the Main Result}
\label{sec:proof}
Fix $\alpha\in 2\NN+1$. Let $B=B_n$ from Theorem~\ref{th-blaschke} with $n=\alpha$.  By standard considerations of continuity and monotonicity of the rotation number,  for every irrational $\rho\in\TT$ there exists a unique $\theta\in\TT$ such that setting 
$$B_\rho\equiv e^{2\pi i\theta}B,\text{ we have }\rho(B_\rho|_{S^1})=\rho.$$
For ease of reference, let us formulate:
\begin{prop}
\label{prop:unique}
Suppose $B_\rho$ is as above, and let $G$ be any other Blaschke fraction with critical points at $0$, $1$, $\infty$, which is topologically conjugate to $B_\rho$ on $\hat\CC$. Then
$$G=B_\rho.$$
\end{prop}
By Theorem~\ref{complex bounds}, there exists $N\in\NN$ such that for all $m\geq N$ and $\rho\notin\QQ$, the renormalization $\cR^mB_\rho$ extends to a holomorphic commuting pair $\cH$ in $\hol^{n}(\mu({n}))$. 

\begin{thm}
\label{th:deform}
Suppose $\cH$ is a holomorphic commuting pair as above. Then every Beltrami differential $\mu$ which is $S^1$-symmetric, invariant under $S_{\cH}$ satisfies
$$\mu=0\text{ a.e. on }K(\cH).$$
\end{thm}
\begin{proof}
Let $\mu_1$ be the Beltrami differential which is given by $\mu$ on $K(\cH)$, and the trivial differential $\sigma_0$ elsewhere on $\Delta_\cH$. By invariance of $K({\cH})$, we have
$$S^*_\cH\mu_1=\mu_1\text{ a.e.}$$
Note that $K(\cH)=J(\cH)$ (cf. \cite{Yamp-towers}).
Extend $\mu_1$ by the dynamics of $B_\rho$ to a $B_\rho$-invariant, $S^1$-symmetric Beltrami differential supported on 
$$\overline{\cup_{k\geq 0}B_\rho^{-k}(K(\cH))}=J(B_\rho)$$
(the equality is implied by the standard properties of Julia sets, cf. e.g. \cite{Mil}). Complete it to a $B_\rho$-invariant Beltrami diffential $\lambda_1$ in $\hat\CC$ by setting it equal to $\sigma_0$ on the complement of the Julia set of $B_\rho$. For $t\in[0,1]$, set
$$\lambda_t=t\lambda_1,$$
so that $\lambda_0=\sigma_0$. By Ahlfors-Bers-Boyarski Theorem, there exists a unique solution $\psi_t:\hat\CC\to\hat\CC$ of the Beltrami equation
$$\psi_t^*\sigma_0=\lambda_t$$
which fixes the points $0$, $1$, and $\infty$ and continuously depends on $t$. By $S^1$-symmetry, the rational map
$$G\equiv \psi_t\circ B_\rho\circ\psi_t^{-1}$$
is a Blaschke fraction. By 
Proposition~\ref{prop:unique}, 
$$G=B_\rho.$$
Hence, for every $t$ and $m$, the map $\psi_t$ permutes the discrete set $B_\rho^{-m}(1)$. Since $\psi_0=\text{Id}$ and $\psi_t$ continuously depends on $t$, $\psi_t$ fixes the points in each  $B_\rho^{-m}(1)$. By the standard properties of Julia sets,
$$\overline{\cup_{m\geq 0}B_\rho^{-m}(1)}=J(B_\rho).$$
Hence, for all $t\in[0,1]$, 
$$\psi_t|_{J(B_\rho)}=\text{Id}.$$
By Bers Sewing Lemma,
$$\bar\partial\psi_t=t\lambda_1=0\text{ a.e.}$$

\end{proof}

We now conclude:
\begin{proof}[Proof of Theorem \ref{thm:main}]
By Complex Bounds, there exist $\rho\notin\QQ$ and $k\in\NN$ such that the renormalization $\cR^k(B_\rho)$ extends to a holomorphic commuting pair $\cG$ which is $K$-quasiconformally conjugate to $\cH$. Denote $\psi$ such a conjugacy. The claim will follow from the above result if we can show that $\bar\partial\psi=0$ a.e. on $K(\cH)$. Indeed, assume the contrary. Then 
$(\psi^{-1})^*\sigma_0$ is a non-trivial invariant Beltrami differential on $K(\cG)$, which contradicts Theorem~\ref{th:deform}.
\end{proof}

\bibliographystyle{amsalpha}
\bibliography{biblio}

\providecommand{\bysame}{\leavevmode\hbox to3em{\hrulefill}\thinspace}
\providecommand{\MR}{\relax\ifhmode\unskip\space\fi MR }
\providecommand{\MRhref}[2]{%
  \href{http://www.ams.org/mathscinet-getitem?mr=#1}{#2}
}
\providecommand{\href}[2]{#2}
\begin{thebibliography}{dFdM00}

\bibitem[dF99]{dF2}
Edson de~Faria, \emph{Asymptotic rigidity of scaling ratios for critical circle
  mappings}, Ergodic Theory Dynam. Systems \textbf{19} (1999), no.~4,
  995--1035.

\bibitem[dFdM99]{dFdM1}
E.~de~Faria and W.~de~Melo, \emph{Rigidity of critical circle mappings {I}}, J.
  Eur. Math. Soc. (JEMS) \textbf{1} (1999), no.~4, 339--392.

\bibitem[dFdM00]{dFdM2}
\bysame, \emph{Rigidity of critical circle mappings {II}}, J. Amer. Math. Soc.
  (JAMS) \textbf{13} (2000), no.~2, 343--370.

\bibitem[FKS82]{FKS}
M.~J. Feigenbaum, L.~P. Kadanoff, and S.~J. Shenker, \emph{Quasiperiodicity in
  dissipative systems: a renormalization group analysis}, Phys. D \textbf{5}
  (1982), no.~2-3, 370--386.

\bibitem[Mil06]{Mil}
J.~Milnor, \emph{Dynamics in one complex variable. {I}ntroductory lectures},
  3rd ed., Princeton University Press, 2006.

\bibitem[{\"O}RSS83]{ORSS}
Stellan {\"O}stlund, David Rand, James Sethna, and Eric Siggia, \emph{Universal
  properties of the transition from quasiperiodicity to chaos in dissipative
  systems}, Phys. D \textbf{8} (1983), no.~3, 303--342.

\bibitem[Vie15]{vieira}
A.~Vieira, \emph{Pares holomorfos e a familia de {A}rnol d generalizada}, Ph.D.
  thesis, Universidade de {S}{\~a}o {P}aulo, 12 2015.

\bibitem[Yam99]{Ya1}
M.~Yampolsky, \emph{Complex bounds for renormalization of critical circle
  maps}, Erg. Th. \& Dyn. Systems \textbf{19} (1999), 227--257.

\bibitem[Yam01]{Yamp-towers}
\bysame, \emph{The attractor of renormalization and rigidity of towers of
  critical circle maps}, Comm. {M}ath. {P}hys. \textbf{218} (2001), no.~3,
  537--568.

\bibitem[Yam02]{Ya3}
M~Yampolsky, \emph{Hyperbolicity of renormalization of critical circle maps},
  Publ. Math. Inst. Hautes {\'E}tudes Sci. \textbf{96} (2002), 1--41.

\bibitem[Yam03]{Ya4}
M.~Yampolsky, \emph{Global renormalization horseshoe for critical circle maps},
  Commun. {M}ath. {P}hys. \textbf{240} (2003), 75--96.

\end{thebibliography}
\end{document}